\documentclass[12pt]{article}
\usepackage[amsmath,cleveref]{e-jc}
\usepackage{graphicx}
\usepackage{tikz}
\MSC{05C35, 05C65}
\title{An Upper Bound on the Linear Tur\'{a}n Number of
$k$-Crowns}
\author{Rajat Adak}

\authortext{}{Department of Computer Science and Automation, Indian Institute of Science, Bengaluru, India (\email{rajatadak@iisc.ac.in}).}
\date{}

\begin{document}

\maketitle
\begin{abstract}
    A hypergraph $H$ is said to be \emph{linear} if every pair of vertices lies in at most one hyperedge.
Given a family $\mathcal{F}$ of $r$-uniform hypergraphs (also called $r$-graphs), an $r$-graph $H$ is said to be \emph{$\mathcal{F}$-free} if it contains no member of $\mathcal{F}$ as a subhypergraph.
The \emph{linear Tur\'{a}n number} $ex_r^{\mathrm{lin}}(n,\mathcal{F})$ denotes the maximum number of edges in an $\mathcal{F}$-free linear $r$-graph on $n$ vertices.

The crown is a linear $3$-graph obtained from three pairwise disjoint edges by adding an edge that intersects each of them in a distinct vertex. Recently, Gy\'arf\'as, Ruszink\'o, and S\'ark\"ozy~[\emph{Linear Tur\'an numbers of acyclic triple systems}, European J.\ Combin.\ (2022)] initiated the study of bounds on the linear Tur\'an number for acyclic $3$-uniform linear hypergraphs, including that of the crown.

We extend the notion of a crown by defining a $k$-crown, denoted by $C_{1,k}^r$, to be a linear $r$-graph consisting of one base edge together with $k$ pairwise disjoint edges, each intersecting the base in a distinct vertex. In this paper, we establish an upper bound on $ex_r^{\mathrm{lin}}(n,C_{1,k}^r)$, which in particular improves the recent bound of Zhang, Broersma, and Wang~[\emph{Generalized Crowns in Linear $r$-Graphs}, Electron.\ J.\ Combin.\ (2025)] for all $r \geq 4$, without forbidding any auxiliary configuration. We also note that the cases $k\in\{1,2\}$ correspond to the short linear paths $P_2^r$ and $P_3^r$, and can be treated separately.
\end{abstract}
\section{Introduction}
Extremal problems in hypergraphs form a central part of combinatorics, with Tur\'an-type questions asking for the maximum number of edges in a hypergraph avoiding a given forbidden configuration. In the setting of linear hypergraphs, these problems often exhibit behavior quite different from that of general hypergraphs, since the linearity condition imposes strong structural restrictions. Recall that a hypergraph $H$ is called \emph{linear} if every pair of vertices is contained in at most one edge. For a family $\mathcal{F}$ of $r$-uniform hypergraphs, the \emph{linear Tur\'an number} $ex_r^{\mathrm{lin}}(n,\mathcal{F})$ denotes the maximum number of edges in an $\mathcal{F}$-free linear $r$-graph on $n$ vertices.

The study of linear Tur\'an numbers for acyclic hypergraphs has recently attracted considerable attention. In particular, Gy\'arf\'as, Ruszink\'o, and S\'ark\"ozy~\cite{gyarfas2022linear} initiated a systematic study of linear Tur\'an numbers of acyclic $3$-uniform hypergraphs. Among the configurations considered in their work is the \emph{crown}, a linear $3$-graph consisting of three pairwise disjoint edges together with a fourth edge, called the \textit{base}, intersecting each of them in exactly one vertex. This configuration is a natural and fundamental example of an acyclic linear hypergraph, and its extremal behavior has motivated further developments in the area.

In \cite{fletcher2021improved}, the author improved the bound for crowns, which was later further improved in \cite{tang2022turan}. More recently, Zhang, Broersma, and Wang~\cite{zhang2025generalized} introduced generalized crowns in linear $r$-graphs and obtained upper bounds on their linear Tur\'an numbers. Their work extends the investigation of crown-type configurations beyond the $3$-uniform setting. Motivated by these developments, it is natural to study broader crown-like structures in linear $r$-graphs and to seek sharper bounds for their linear Tur\'an numbers.

In this paper, we extend the notion of a crown to linear $r$-graphs by defining a $k$-crown, denoted by $C_{1,k}^r$, to be a linear $r$-graph consisting of one base edge together with $k$ pairwise disjoint edges, each intersecting the base in a distinct vertex. In the crown setting considered here, we assume $3\leq k \leq r$. This notion provides a natural generalization of the usual crown and allows one to treat a wider class of crown-type configurations in a unified way. The degenerate cases $k\in\{1,2\}$ correspond to the short linear paths $P_2^r$ and $P_3^r$, respectively; we treat these separately.

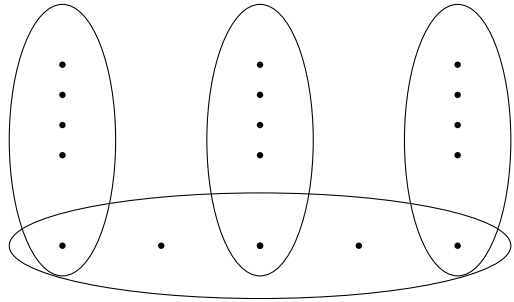
\begin{figure}[ht]
    \centering
    \begin{tikzpicture}[scale=1,
    every node/.style={circle,fill=black,inner sep=1.5pt}
]

\begin{scope}[xshift=7.5cm]

% Ellipses
\draw (0,0) ellipse (0.7 and 1.8);
\draw (2.6,0) ellipse (0.7 and 1.8);
\draw (5.2,0) ellipse (0.7 and 1.8);
\draw (2.6,-1.4) ellipse (3.3 and 0.7);

% Dots in vertical ellipses
\foreach \x in {0,2.6,5.2} {
  \foreach \y in {1,0.6,0.2,-0.2,-1.4}
    \fill (\x,\y) circle (1.2pt);
}

% Dots in bottom horizontal ellipse
\foreach \x in {1.3,2.6,3.9}
  \fill (\x,-1.4) circle (1.2pt);

\end{scope}
\end{tikzpicture}
\caption{Configuration of $C_{1,3}^5$ ($k =3$, $r =5$)}
\end{figure}

Our main result gives an upper bound on $ex_r^{\mathrm{lin}}(n,C_{1,k}^r)$. In particular, when $k=r$, it improves the recent upper bound of Zhang, Broersma, and Wang~\cite{zhang2025generalized} for all $r \geq 4$. Moreover, our result requires only the exclusion of $C_{1,r}^r$, whereas the approach in~\cite{zhang2025generalized} additionally forbids other auxiliary configurations. Thus, our result strengthens the currently known bounds for generalized crowns in linear $r$-graphs.

In \cite{gyarfas2022linear}, Gy\'arf\'as, Ruszink\'o, and S\'ark\"ozy initiated the study of Linear Tur\'{a}n number for crowns with the following result.
\begin{theorem}[\cite{gyarfas2022linear}, Theorem 1.7]\label{thm1} $$6\left\lfloor\frac{n-3}{4}\right\rfloor + \epsilon \leq ex_3^{\mathrm{lin}}(n,C^3_{1,3}) \leq 2n,$$ where $\epsilon =0$, if $n-3 \equiv 0,1$ $(\text{mod } 4)$, $\epsilon =1$ if $n-3 \equiv 2$ $(\text{mod } 4)$, and $\epsilon =3$ if $n-3 \equiv 3$ $(\text{mod } 4)$.
    
\end{theorem}
Fletcher~\cite{fletcher2021improved} improved the upper bound of \Cref{thm1} to $\frac{5n}{3}$. Later, Carbonero, Fletcher, Guo, Gy\'arf\'as, Wang, and Yan in \cite{carbonero2022crowns} proved that every linear $3$-graph with minimum degree $4$ contains a crown. They also conjectured that $ex(n,C^3_{1,3}) \sim \frac{3n}{2}$, in \cite{carbonero2021crowns}.

Later, Tang, Wu, Zhang, and Zheng proved this conjecture in~\cite{tang2022turan}. They gave the following result.
\begin{theorem}[\cite{tang2022turan}, Theorem 1]\label{thm2}
$$ex_3^{\mathrm{lin}}(n,C^3_{1,3}) \leq \frac{3(n-s)}{2},$$ where $s$ is the number of vertices of degree at least $6$.
\end{theorem}
\Cref{thm2} was later extended to $r$-graphs by Zhang, Broersma, and Wang~\cite{zhang2025generalized}. For this, they introduced another generalization of the crown. They defined $C_{1,r}^*$ to be the linear $r$-graph on $r^2-r+3$ vertices and $r+1$ edges consisting of $r-2$ edges $e_1,e_2,\dots,e_{r-2}$ that intersect in exactly one common vertex $v$, two additional disjoint edges $e_{r-1}$ and $e_r$ that are also disjoint from $\{e_1,e_2,\dots,e_{r-2}\}$, and one further edge $e$ intersecting each edge of $\{e_1,e_2,\dots,e_r\}$ in exactly one vertex other than $v$. Instead of forbidding just $C^r_{1,r}$, they gave bounds for $\{C^r_{1,r},C^*_{1,r}\}$-free linear $r$-graphs. They proved that,
\begin{theorem}[\cite{zhang2025generalized}, Theorem 3]\label{thm3}
$$r(r-1)\left\lfloor \frac{n-r}{(r-1)^2} \right\rfloor
\le ex_r^{\mathrm{lin}}(n, \{C^r_{1,r},C^*_{1,r}\}) \le
\frac{r(r-2)(n-s)}{r-1},$$
    where $s$ is the number of vertices with degree at least $(r-1)^2+2$.
\end{theorem}
Note that for $r =3$, $C_{1,r}^r\cong C^*_{1,r}$. Thus, \Cref{thm3} is a generalization of \Cref{thm2}.

Recently, Adak and Verma~\cite{adak2026bounds} generalized the bounds for trees with four edges in linear $3$-graphs established in~\cite{gyarfas2022linear} to the setting of linear $r$-graphs. They established the following result.
\begin{theorem}[\cite{adak2026bounds}, Theorem 4]\label{thm4} For $r \geq 3$, $ex_r^{\mathrm{lin}}(n,C^r_{1,3}) \leq \dfrac{(2r-1)n}{r}$.
\end{theorem}
\section{Main Result}
Our result provides an upper bound on $ex_r^{\mathrm{lin}}(n, C^r_{1,k})$ for $3 \le k \le r$. In particular, when $k = r$ we obtain an improvement over the upper bound in \Cref{thm3} for all $r \geq 4$, while setting $k = 3$ recovers the bound of \Cref{thm4}. 
The cases $k\in \{1,2\}$ fall outside the genuine crown regime and will be discussed separately.
\begin{theorem}\label{main}
    For $3\leq k \leq r$,
    $$ex_r^{\mathrm{lin}}(n,C^r_{1,k}) \leq \frac{\big((k-1)(r-1)+1\big)(n-s)}{r},$$
    where $s$ is the number of vertices with degree at least $(k-1)(r-1) +2$. 
\end{theorem}

\begin{remark}
Taking $k=r$ in \Cref{main}, we obtain
$$ex_r^{\mathrm{lin}}(n,C_{1,r}^r)
\le \frac{((r-1)^2+1)(n-s)}{r},$$
where $s$ denotes the number of vertices of degree at least $(r-1)^2+2$. Thus, the threshold defining $s$ agrees with that in \Cref{thm3}. Now,
$$
\frac{r(r-2)(n-s)}{r-1}-
\frac{((r-1)^2+1)(n-s)}{r}=
\frac{(n-s)(r^2-4r+2)}{r(r-1)}.
$$
Since $r^2-4r+2>0$ for all $r\ge 4$, it follows that
$$
\frac{((r-1)^2+1)(n-s)}{r}<
\frac{r(r-2)(n-s)}{r-1}\qquad \text{for all } r\ge 4.
$$
Moreover, every $\{C_{1,r}^r,C_{1,r}^*\}$-free linear $r$-graph is, in particular, $C_{1,r}^r$-free. Hence for $r \geq 4$,
$$
ex_r^{\mathrm{lin}}(n,\{C_{1,r}^r,C_{1,r}^*\})
\le
ex_r^{\mathrm{lin}}(n,C_{1,r}^r)
\le
\frac{((r-1)^2+1)(n-s)}{r}
<
\frac{r(r-2)(n-s)}{r-1}.
$$
Therefore, for $r\ge 4$, \Cref{main} yields a better upper bound than \Cref{thm3}.

Note that for $r=3$, the bound in \Cref{thm2} remains stronger.
\end{remark}
The improvement is noteworthy for two reasons. First, it yields a sharper coefficient even though Zhang, Broersma, and Wang work under the stronger assumption of forbidding both $C_{1,r}^r$ and the auxiliary configuration $C_{1,r}^*$. Second, our proof uses only the crown-free assumption and a direct degree-summation argument, and therefore avoids the additional structural analysis required by auxiliary forbidden configurations.
\begin{remark}
    Taking $k =3$, we get
    $$ex_r^{\mathrm{lin}}(n,C_{1,3}^r) \leq \frac{(2(r-1)+1)(n-s)}{r} \leq \frac{(2r-1)n}{r},$$
    where $s$ is the number of vertices with degree at least $6$ and thus trivially $s \geq 0$. Therefore, we recover the bound from \Cref{thm4}.
\end{remark}
\section{Proof of \Cref{main}}
Let $H$ be a linear $C^r_{1,k}$-free $r$-graph on $n$ vertices. We can assume $H$ does not contain any isolated vertices since deleting isolated vertices does not change the edge set.
\begin{lemma}\label{lem1}
    Let $G$ be a linear $r$-graph. If there exists $e \in E(G)$, with distinct vertices $v_1,v_2,\dots,v_k \in V(e)$, such that $d(v_i) \geq (i-1)(r-1)+2$ for all $i \in [k]$, then $G$ contains a copy of $C^r_{1,k}$ with $e$ as its base.
\end{lemma}
\begin{proof}
We choose the $k$ edges of $C_{1,k}^r$ other than the base edge $e$ greedily, selecting one edge through each of the vertices $v_1,v_2,\dots,v_k$. Since $d(v_1)\ge 2$, there exists an edge $e_1 \in E(G)$, distinct from $e$, such that $v_1 \in e_1$.

Now suppose inductively that for some $i \ge 2$, we have already chosen pairwise disjoint edges $e_1,e_2,\dots,e_{i-1}$ such that $e_j$ intersects $e$ exactly at $v_j$ for each $j\in [i-1]$.

Observe that the set $\bigcup_{j=1}^{i-1}(V(e_j)\setminus\{v_j\})$ contains $(r-1)(i-1)$ vertices. Consider the edges incident on $v_i$. There are at least $(r-1)(i-1)+1$ edges through $v_i$ distinct from $e$. By linearity, none of these edges intersects $e$ except at $v_i$. Moreover, each $e_j$, $j \in [i-1]$, contains $r-1$ vertices outside $e$, and hence by linearity at most $r-1$ edges through $v_i$ (other than $e$) can intersect any fixed $e_j$.

Therefore, at most $(r-1)(i-1)$ edges through $v_i$ intersect one of the previously chosen edges $e_1,\dots,e_{i-1}$. Consequently, there exists at least
$$
(d(v_i)-1)-(r-1)(i-1)\ge 1
$$
edge through $v_i$ distinct from $e$ that is disjoint from $e_1,\dots,e_{i-1}$. Choose such an edge and denote it by $e_i$.

Proceeding inductively, we obtain edges $e_1,\dots,e_k$ such that $v_i\in e_i$ for every $i\in[k]$, the edges $e_1,\dots,e_k$ are pairwise disjoint, and each $e_i$ intersects $e$ exactly in $v_i$. Hence $e,e_1,\dots,e_k$ form a copy of $C_{1,k}^r$ with base edge $e$.
\end{proof}

\begin{lemma}\label{lem2}
    Let $e = \{u_1,u_2,\dots u_r\}$ be any arbitrary edge in $H$ with $d(u_1) \geq d(u_2) \geq\dots \geq d(u_r)$. Then there exists $i \in [k]$ such that $d(u_i) \leq (k-i)(r-1) +1$.
\end{lemma}
\begin{proof}
    Suppose on the contrary, $d(u_i) \geq (k-i)(r-1) +2$ for all $ i \in [k]$. Set $v_i = u_{k-i+1}$ for all $i \in [k]$. Thus we get,
    $$d(v_i) = d(u_{k-i+1}) \geq (i-1)(r-1) +2 \qquad \text{for all $i \in [k]$}$$
    Hence, by Lemma~\ref{lem1}, the edge $e$ is the base of a copy of
$C_{1,k}^r$, contradicting our assumption that $H$ is $C_{1,k}^r$-free.
This proves the lemma.
\end{proof}
Let $e \in E(H)$ with $e = \{u_1,u_2,\dots,u_r\}$ and $d(u_1) \geq d(u_2)\geq \dots \geq d(u_r)$. From Lemma~\ref{lem2}, let $i \in [k]$ be such that $d(u_i) \leq (k-i)(r-1) +1$. 

Since the degrees are non-increasing, we have $d(u_j) \leq (k-i)(r-1) +1$ for all $j \in \{i, i+1,\dots, r\}$. In particular, the $r-i+1$ vertices $u_i,u_{i+1},\dots,u_{r}$ have degree at most $(k-i)(r-1)+1$.

Recall that $s$ is the number of vertices with degree at least $(k-1)(r-1)+2$. Now define the indicator function $I: V(H) \rightarrow \{0,1\}$. Such that,
$$I(v)=
\begin{cases}
1, & \text{if $d(v) \leq (k-1)(r-1)+1$},\\
0,& \text{otherwise}
\end{cases}$$
Therefore,  we have \begin{equation}\label{eqn1}
    \sum_{v \in V(H)} I(v) = n-s.
\end{equation}

Since $(k-1)(r-1)+1 \geq (k-i)(r-1)+1$ for all $ i\in [k]$, for the edge $e\in E(H)$ and the choice of $i \in [k]$ we get,
\begin{equation}\label{eqn2}
    \sum_{v \in e} \frac{I(v)}{d(v)} \geq \sum_{j=i}^r\frac{1}{d(u_j)} \geq \frac{r-i+1}{(k-i)(r-1)+1}
\end{equation}
Note that, 
\begin{equation*}
    \frac{r-i+1}{(k-i)(r-1)+1} - \frac{r}{(k-1)(r-1)+1} = \frac{(i-1)\big(r(r-k)+k-2\big)}{\big((k-i)(r-1)+1\big)\big((k-1)(r-1)+1\big)}
\end{equation*}
Since $3\leq k \leq r$, we have $(r(r-k)+k-2) > 0$. Also, $i \geq 1$. Therefore,
\begin{equation}\label{eqn3}
    \frac{r-i+1}{(k-i)(r-1)+1} \geq \frac{r}{(k-1)(r-1)+1}
\end{equation}
Therefore, from \Cref{eqn2,eqn3} we get that, for all $e \in E(H)$,
\begin{equation*}
    \sum_{v \in e}\frac{I(v)}{d(v)} \geq \frac{r}{(k-1)(r-1)+1}
\end{equation*}
Thus summing over all the edges gives that,
\begin{equation}\label{eqn4}
    \sum_{e \in E(H)}\sum_{v \in e}\frac{I(v)}{d(v)} \geq |E(H)|\frac{r}{(k-1)(r-1)+1}
\end{equation}
Reversing the order of summation and from \Cref{eqn1} we get,
\begin{equation}\label{eqn5}
     \sum_{e \in E(H)}\sum_{v \in e}\frac{I(v)}{d(v)} =\sum_{v \in V(H)} \sum_{e \ni v}\frac{I(v)}{d(v)} = \sum_{v \in V(H)} d(v) \frac{I(v)}{d(v)} =  \sum_{v \in V(H)}I(v) = n-s
\end{equation}
Thus, from \Cref{eqn4,eqn5} we get,
\begin{align}
    &n-s \geq |E(H)|\frac{r}{(k-1)(r-1)+1}\nonumber\\
    \implies & |E(H)| \leq \frac{\big((k-1)(r-1)+1\big)(n-s)}{r}\label{eqn6}
\end{align}
Since the bound in \Cref{eqn6} is true for any arbitrary linear $C_{1,k}^r$-free $r$-graph $H$ on $n$ vertices, we get that
$$ex_r^{\mathrm{lin}}(n,C^r_{1,k}) \leq \frac{\big((k-1)(r-1)+1\big)(n-s)}{r},$$
    where $s$ is the number of vertices with degree at least $(k-1)(r-1) +2$. 

    This completes the proof. \qed
\vspace{2mm}

We next prove a weighted refinement of \Cref{main}.

\begin{definition}
Let $H$ be a hypergraph and let $e\in E(H)$. We define $k(e)$ to be the maximum integer $t$ such that $e$ is the base edge of a copy of $C_{1,t}^r$ in $H$.
\end{definition}

\begin{proposition}\label{weighted}
Let $H$ be a linear $C_{1,k}^r$-free $r$-graph on $n$ vertices, and let $s$ denote the number of vertices of degree at least $(k-1)(r-1)+2$. Then
\[
\sum_{e\in E(H)} \frac{1}{k(e)(r-1)+1} \le \frac{n-s}{r}.
\]
\end{proposition}

\begin{proof}
We use the same notation and the same double-counting argument as in the proof of \Cref{main}. Let $I$ be the indicator function defined there. Then
$\sum_{v\in V(H)} I(v)=n-s$.

Fix an edge $e=\{u_1,\dots,u_r\}\in E(H)$ with $d(u_1)\ge d(u_2)\ge \cdots \ge d(u_r)$,
and set $t:=k(e)+1$. Since $e$ is not the base edge of any copy of $C_{1,t}^r$, the argument of Lemma \ref{lem2} yields an index $i\in [t]$ such that
\[
d(u_i)\le (t-i)(r-1)+1.
\]
Hence, exactly as in the proof of \Cref{main},
\[
\sum_{v\in e}\frac{I(v)}{d(v)}
\ge \frac{r-i+1}{(t-i)(r-1)+1}
\ge \frac{r}{(t-1)(r-1)+1}
= \frac{r}{k(e)(r-1)+1}.
\]
Summing over all edges and reversing the order of summation, we obtain
\[
r\sum_{e\in E(H)}\frac{1}{k(e)(r-1)+1}
\le
\sum_{e\in E(H)}\sum_{v\in e}\frac{I(v)}{d(v)}
=
\sum_{v\in V(H)} I(v)
=
n-s.
\]
This proves the proposition.
\end{proof}

Since $k(e)\le k-1$ for every edge $e\in E(H)$, we have
$k(e)(r-1)+1\le (k-1)(r-1)+1$,
and hence \Cref{main} follows immediately from Proposition \ref{weighted}. Thus, Proposition \ref{weighted} may be viewed as a weighted refinement of \Cref{main}.

This type of local weighted refinement is often referred to as \emph{localization}, and such refinements have been studied extensively in recent years in connection with classical extremal bounds. For related weighted refinements results, see \cite{adak2025vertex,adak2026localization,malec2023localized,zhao2025localized}.

\section{Extending \Cref{main} to smaller values of $k$}
We separate this section from the rest of the analysis since for $k\in \{1,2\}$, $C^r_{1,k}$ falls outside the genuine crown regime.
Indeed, in our definition of $C_{1,k}^r$, we work with $3 \le k \le r$, so that
$C_{1,k}^r$ is a true crown-type configuration consisting of a base edge together with
$k$ pairwise disjoint edges meeting the base in distinct vertices. By contrast, $C^r_{1,1} \cong P^r_2$ and $C^r_{1,2} \cong P^r_3$, where $P^r_{i}$ is a linear $r$-path on $i$ edges.

Indeed, throughout the proof of \Cref{main}, the assumption $k\ge 3$ is used only
to verify that $r(r-k)+k-2>0$.
Since $r\ge 3$, this inequality also holds for $k\in\{1,2\}$. Therefore the same
argument yields the bound of \Cref{main} for $k=1$ and $k=2$.
\begin{remark}
    Taking $k =1$, we get
    $ex_r^{\mathrm{lin}}(n,P^r_2) \leq \frac{n-s}{r}$, where $s$ is the number of vertices with degree at least $2$. But we cannot have any vertex with degree more than $1$, otherwise we will get a $P^r_2$, therefore $s=0$. Thus we get,
    $$ex_r^{\mathrm{lin}}(n,P^r_2) \leq \left\lfloor\frac{n}{r}\right\rfloor$$ This bound is tight, as it is attained by a maximum matching.
\end{remark}

\begin{remark}
Taking $k=2$ in the bound of \Cref{main}, we obtain
$$ex_r^{\mathrm{lin}}(n,P_3^r)\le n-s,$$
where $s$ is the number of vertices of degree at least $r+1$.
On the other hand, \cite{zhang2025linear} proved that $ex_r^{\mathrm{lin}}(n,P_3^r)\le n$,
 and characterized the extremal graph class. In particular, the extremal graph is $r$-regular, and hence
$s=0$ in the equality case. Thus our bound may be viewed as a refinement in terms of $s$.
Interestingly, it is sharp only when $s=0$.

\medskip
\noindent\textbf{Claim.}
If $s>0$, then the bound $ex_r^{\mathrm{lin}}(n,P_3^r)\le n-s$ is not tight.

\begin{proof}
Suppose, for a contradiction, that $H$ is a linear $P_3^r$-free $r$-graph on $n$ vertices with
$|E(H)|=n-s,$ where $s>0$ is the number of vertices of degree at least $r+1$.

Let $v \in V(H)$ such that $d(v) \geq r+1$. Let $e_1,\dots,e_{d(v)}$ be the edges containing $v$. Without loss of generality, suppose that some vertex
$x\in e_1\setminus\{v\}$ lies in an edge $f$ not containing $v$. By linearity,
the edge $f$ can meet at most $r-1$ of the edges $e_2,\dots,e_{d(v)}$. Since
$d(v)\ge r+1$, there exists some $j\ge 2$ such that $f\cap e_j=\varnothing$.
Then the three edges $f,e_1,e_j$ form a copy of $P_3^r$. Therefore, the component containing the vertex $v$ must be a star centered at $v$, otherwise we will get a $P_3^r$. 

Therefore, each of the $s$ vertices of degree at least $r+1$ is the center of a distinct star
component. Let these components be $S_1,\dots,S_s$, and let $v_i$ be the center of $S_i$.
Set $H' = H \setminus \bigcup_{i=1}^s S_i$.
Since every vertex of $H'$ has degree at most $r$, we have $$r|E(H')|=\sum_{u\in V(H')} d(u)\le r|V(H')| \implies |E(H')|\le |V(H')|$$
Since, $S_i$ is a star and $d(v_i) \geq r+1$ we have,
$$|V(S_i)|-|E(S_i)| = 1+d(v_i)(r-1) - d(v_i) \ge 1+(r+1)(r-2)$$
Consequently,
$$
\begin{aligned}
s
= n-|E(H)|
&= \bigl(|V(H')|-|E(H')|\bigr)+\sum_{i=1}^s \bigl(|V(S_i)|-|E(S_i)|\bigr) \\
&\ge 0+s\bigl(1+(r+1)(r-2)\bigr)
> s,
\end{aligned}
$$
a contradiction.
\end{proof}
\end{remark}

\section*{Concluding Remarks}

We introduced the notion of a $k$-crown $C_{1,k}^r$ in linear $r$-graphs and established the upper bound
\[
ex_r^{\mathrm{lin}}(n,C_{1,k}^r)\le \frac{((k-1)(r-1)+1)(n-s)}{r},
\]
where $s$ denotes the number of vertices of degree at least $(k-1)(r-1)+2$. 
Although our main focus is on the genuine crown regime $3\le k\le r$, we also observed that the same argument extends formally to the degenerate cases $k\in \{1,2\}$, which correspond to the short linear paths $P_2^r$ and $P_3^r$, and therefore were treated separately.

In particular, when $k=r$, our bound improves the upper bound of Zhang, Broersma, and Wang~\cite{zhang2025generalized} for all $r\ge 4$, while for $k=3$ it recovers the bound for four-edge trees from \cite{adak2026bounds}.

An interesting direction for future work is to determine whether the bound in \Cref{main} is asymptotically sharp for fixed $r$ and $k$. In particular, it would be interesting to obtain matching lower bounds for $C_{1,k}^r$-free linear $r$-graphs, and to understand whether sharper extremal estimates can be proved by forbidding only $C_{1,k}^r$, without introducing auxiliary crown-type configurations.
\bibliographystyle{plain}
\bibliography{references}
\end{document}